\documentclass[graybox]{svmult}

\usepackage{mathrsfs, amssymb, eucal, amsmath, amscd}
\usepackage{mathptmx}       % selects Times Roman as basic font
\usepackage{helvet}         % selects Helvetica as sans-serif font
\usepackage{courier}        % selects Courier as typewriter font
\usepackage{type1cm}        % activate if the above 3 fonts are
\usepackage{makeidx}         % allows index generation
\usepackage{graphicx}        % standard LaTeX graphics tool
\usepackage{multicol}        % used for the two-column index
\usepackage[bottom]{footmisc}% places footnotes at page bottom

\makeindex             % used for the subject index
\begin{document}

\title*{On $(h,k)-$ dichotomy of linear discrete-time systems in Banach spaces}
\author{ Violeta Crai and Mirela Aldescu}
\institute{ Violeta Crai \at Department of Mathematics
West University of Timi\c soara,
4 V. P\^arvan Blvd.
300223 Timi\c soara
Romania \email{vio.terlea@gmail.com}
\and  Mirela Aldescu \at  Department of Mathematics
West University of Timi\c soara,
4 V. P\^arvan Blvd.
300223 Timi\c soara
Romania \email{mirelaaldescu@yahoo.com}}
\maketitle

\abstract*{The paper considers a general concept of dichotomy with different growth rates for linear discrete-time systems in Banach spaces. Characterizations in terms of Lyapunov type sequences of norms are given. The approach is illustrated by various examples.}

\abstract{The paper considers a general concept of dichotomy with different growth rates for linear discrete-time systems in Banach spaces. Characterizations in terms of Lyapunov type sequences of norms are given. The approach is illustrated by various examples.}
\section{Introduction}

The notion of (uniform) exponential dichotomy for difference equations was introduced in the literature by T. Li \cite{li} and 
plays a central role in the theory of dynamical system such as, for example, in the study of stable and unstable manifolds and in many aspects of the theory of stability. We note that the theory of exponential dichotomies and its applications are very much developed.

Early results in the study of dichotomies for
difference equations appeared in the paper of C.V. Coffman and J.J. Schaffer \cite{cofman}.
Later, in 1981,  D. Henry included discrete dichotomies in his book \cite{henry}. These
were followed by the monographs due to R.P. Agarwal \cite{argwal}, C. P\"otzsche \cite{potzche} and S. Elaydi \cite{elaydi} (deals with ordinary dichotomy).

Lately, characterizations of the nonuniform exponential
dichotomy for discrete linear systems can be found in the works
of  M. Megan,  T. Ceau\c{s}u,  A.L. Sasu, B. Sasu, L. Popa,  M.G. Babu\c{t}ia   and colleagues (see ~\cite{babutia-megan, megan-sasi, popa}).
 In 2009 A. Bento and C. Silva introduced a new concept of dichotomy called polynomial dichotomy \cite{bento}. N.M. Seimeanu in \cite{seimeanu} studied connections between different concepts of polynomial dichotomy.

 A natural generalization of both exponential and polynomial
dichotomy is successfully modeled by the concept of $(h,k)-$ dichotomy
introduced by M. Pinto \cite{pinto} for invertible difference equations. Two years later M. Megan (\cite{megan}) developed his research  and lately the concept was intensively studied in its various forms: uniform and nonuniform, strong and weak ~\cite{babutia-megan-popa, vio-saci}.  
In this paper we consider the general and more realistic case of a non-invertible dynamics as well as of
a nonuniform $(h,k)-$ dichotomy (this concept is a generalization of a dichotomy concept studied by L. Barreira and C. Valls in \cite{ba0}).
Both notions are ubiquitous in ergodic theory. 

Our approach consists in reducing
the study of the dynamics to one with uniform  behavior with
respect to a sequence of norms.  
This allows one to take profit of the existing theory and methods for a dynamics
with a uniform behavior in order to transfer some of the information
to the original dynamics.

It is difficult to indicate an original reference for considering sequences of
norms in the classical uniform theory (both for discrete and continuous time), but in
the nonuniform theory it first occurred in Pesin's work \cite{pesin}. 
 L. Barreira, D. Dragi\v{c}evi\'{c} and C. Valls \cite{ba0} obtained characterizations of the concept of exponential dichotomy in terms of admissibility using characterizations with sequences of norms. Our characterizations of $(h,k)-$ dichotomy with respect to a sequence of norms is motivated
by their approach and it is in the same time a generalization. 

The aim of this paper is to give characterizations of the concept using sequences of norms. The most important result is the equivalence between the concept of $(h,k)-$ dichotomy and a certain type of uniform $(h,k)-$ dichotomy with respect to a sequence of norms. As an application of this result, we give a characterization of Barbashin type  for the concept  and necessary and sufficient conditions of Datko type (see also \cite{claudia-datko-discret}). 
\section{Definitions, examples and counterexamples}
Let $X$ be a Banach space and $\mathcal{B}(X)$ the Banach algebra of all bounded linear operators on $X$. The norms on $X$ and on $\mathcal{B}(X)$ will be denoted by $\|\cdot \|$. The identity operator on $X$ is denoted by $I$. We also denote by
\begin{align*}
\Delta=\{(m,n)\in \mathbb{N}^2: m\geq n\}, \text{ where } \mathbb{N} \text{ denotes the nonnegative integers.}
\end{align*} 
We consider the linear difference system
\begin{align}
(\mathcal{A})&    &x_{n+1}=A_nx_n,
\end{align}
where $A:\mathbb{N}\to\mathcal{B}(X)$ is a given sequence. \par 
For $(m,n)\in \Delta$ we define:

\[ A_m^n=\left\{ \begin{array}{ll} A_{m-1}.....A_n,&\mbox{if $m>n$} \\
I,& \mbox{ if $m=n$}.
\end{array} \right. \]
\begin{definition}
	The map $A:\Delta\to\mathcal{B}(X)$ defined above is called
	the evolution operator associated to the system $(\mathcal{A})$.
\end{definition}
\begin{remark}
	If the sequence $(x_m)$ is a solution of $(\mathcal{A})$, then
	\begin{eqnarray}
	x_m=A_m^nx_n, &\text{for all}& (m,n)\in\Delta.
	\end{eqnarray}
\end{remark} 

\begin{definition}
	A sequence $P:\mathbb{N}\to\mathcal{B}(X),P(n)=P_n$, for all $n\in\mathbb{N}$  is called
	\begin{enumerate}
		\item \textit{projectors sequence} on X if $P^2_n=P_n$  for all $n\in\mathbb{N}$;
		\item \textit{invariant} for $(\mathcal{A})$ if $A_nP_n=P_{n+1}A_n$, for all $n\in \mathbb{N}$;
		\item \textit{strongly invariant} for $(\mathcal{A})$ if is invariant and the restriction of  $A_m^n$ to Ker $P_n$ is an isomorphism from Ker $P_n$ to Ker $P_m$.
	\end{enumerate}
\end{definition}
\begin{example}
	If a sequence $P:\mathbb{N}\to\mathcal{B}(X)$ is a projectors sequence invariant for $(\mathcal{A})$, then the sequence $Q:\mathbb{N}\to\mathcal{B}(X)$ given by
	\begin{equation}
	Q_n=I-P_n
	\end{equation}
	is also a projectors sequence invariant for $(\mathcal{A})$, called the \textit{complementary projectors sequence} of $(P_n)$.
\end{example}
\begin{remark} 
	It is easy to see that $(P_n)$ is invariant for $(\mathcal{A})$ if and only if
	\begin{align*}
	A_m^nP_n=P_mA_m^n, \text{ for all } (m,n)\in\Delta.
	\end{align*}	
\end{remark}
\begin{remark}\label{rem1}
	If the sequence $(P_n)$ is strongly invariant to $(\mathcal{A})$, then there exists $B:\Delta\to\mathcal{B}(X)$, such that $B_m^n$ is an isomorphism from Ker $P_m$ to Ker $P_n$ and 	
	\begin{enumerate}
		\item $A_m^nB_m^nQ_m=Q_m$
		\item $B_m^nA_m^nQ_n=Q_n$
		\item $B_m^nQ_m=Q_nB_m^nQ_m$
		\item $Q_m=B_m^m Q_m=Q_mB_m^mQ_m$
		\item $B_m^{m_0} Q_m=B_n^{m_0}B_m^nQ_m,$
	\end{enumerate}  
	for all $(m,n),(n,m_0)\in\Delta$, where $(Q_n)$ is the complementary projectors sequence of $(P_n).$\par 
	The map $B$ is called the skew-evolution operator associated to the pair
	$(\mathcal{A},P)$.
\end{remark}
\begin{proof}
	See \cite{babutia-megan-continu}.
	\qed 
	\end{proof}
\begin{definition}
	An increasing sequence $h:\mathbb{N}\to[1,\infty)$ is called \textit{a growth rate} if $$\lim\limits_{n\to\infty}h_n=\infty$$.
\end{definition} 
Let $h,k:\mathbb{N}\to[1,\infty)$ be two growth rates and $(P_n)$ a  projectors sequence invariant for $(\mathcal{A})$.        
\begin {definition}\label{def1}     
The pair $(\mathcal{A},P) $ is \textit{ $(h,k)$-dichotomic} (and we denote by $(h,k)-d.$) if there exists a nondecreasing sequence $  d:\mathbb{N}\rightarrow [1,\infty)$ such that
\begin{enumerate}
	\item[$(hd_1)$] $h_m || A_m^nP_nx|| \leq d_n h_n  ||P_nx||$
	\item[$(kd_1)$] $ k_m || Q_nx|| \leq d_m  k_n ||A_m^nQ_nx||,$
\end{enumerate}
for all $ (m,n,x)\in\Delta\times X$, where $(Q_n)$ is the complementary sequence of $(P_n).$
\end{definition}
\begin{remark}
As particular cases of $(h,k)$-dichotomy we have:
\begin{enumerate}
	\item If the sequence $(d_n)$ is constant then we obtain the uniform- $(h,k)$ -dichotomy, denoted by $u.-(h,k)-d.$
	\item For $h_m=e^{\alpha m}, k_m=e^{\beta m} $ where $\alpha,\beta>0$ it  results the  exponential dichotomy concept, denoted by $e.d.$
	\item If $h_m=(m+1)^\alpha, k_m=(m+1)^\beta,$ where $\alpha,\beta>0$ then we obtain the polynomial dichotomy denoted by $p.d.$
\end{enumerate}	
\end{remark}
\begin{remark}\label{rem-uniform-strong-implica-strong}
	If the pair $(\mathcal{A},P)$ is $u.-(h,k)-d.$ then it is $(h,k)-d.$, but the reverse is not always true, as it results from the following example.
\end{remark}
\begin{example}\label{ex-unif-care-nu-este-dicho}
Let $X=\mathbb{R}^2$ endowed with the norm $\|(x_1,x_2)\|=\max\{|x_1|,|x_2|\}$	and the dynamical system $(\mathcal{A})$ defined by the sequence 
\begin{equation*}
A_n= \frac{1+\ln a_{n}}{1+\ln a_{n+1}}\left( \frac{h_n}{h_{n+1}}P_n+\frac{k_{n+1}}{k_n}Q_{n+1}\right) ,
\end{equation*}
where $a,h,k:\mathbb{N}\to[1,\infty)$ are growth rates and the sequences of projectors $P,Q:\mathbb{N}\to\mathcal{B}(X)$ are given by 
\begin{eqnarray*}
	P_n(x_1,x_2)=(x_1+a_nx_2,0)&\text{and }& Q_n(x_1,x_2)=(-a_nx_2,x_2).	
\end{eqnarray*}
It follows that the evolution operator associated to the system $(\mathcal{A})$ is
\begin{eqnarray*}
	A_m^n= \frac{1+\ln a_{n}}{1+\ln a_{m}}\left( \frac{h_n}{h_{m}}P_n+\frac{k_{m}}{k_n}Q_{m}\right) ,&\text{for all}&(m,n)\in\Delta.
\end{eqnarray*}
We observe that 
\begin{align*}
P_mP_nx=P_nx,& & Q_mQ_nx=Q_mx,&&
\|P_nx\|=(a_n+1)\|x\|,& &\|Q_nx\|=a_n\|x\|
\end{align*} 
and we obtain that the complementary sequences of projectors $(P_n),(Q_n)$ are invariant for the system $(\mathcal{A})$.

From
\begin{align*}
h_m	\|A_m^nP_nx\|&=\frac{1+\ln a_{n}}{1+\ln a_{m}}h_n\|P_nx\|\leq(1+\ln a_{n}) h_n\|P_nx\|\\
&\text{and}\\
(1+\ln a_{m})k_n\|A_m^nQ_nx\|&= (1+\ln a_{n})k_m\|Q_nx\|\geq  {k_m}\|Q_{n}x\|,
\end{align*}
for all $(m,n)\in\Delta$, we obtain that the pair $(\mathcal{A},P)$ is $(h,k)-d.$ with $d_n=1+\ln a_{n}$ .

If we assume that the pair $(\mathcal{A},P)$ is $u.-(h,k)-d.$ then there exists a constant $N\geq 1$ such that
\begin{equation*}
1+\ln a_{m}\leq N(1+\ln a_{n}) \text{, for all } (m,n)\in\Delta.
\end{equation*}
Taking $n=0$ and $m\to\infty$ it results a contradiction.
\end{example}
A first result is the characterization of $(h,k)-d.$ with strongly invariant sequences of projectors.
\begin{theorem}\label{thm-dicho cu strong crestere}
Let $(P_n)$ be strongly invariant for $(\mathcal{A})$.
	The pair $(\mathcal{A},P)$ is $(h,k)-d.$ if and only if there exists a nondecreasing sequence $ s:\mathbb{N}\rightarrow [1,\infty)$ such that
	\begin{enumerate}
		\item[$(hd_2)$] $h_m||A_m^nP_nx||\leq s_n  h_n||x||$
		\item[$(kd_2)$] $ k_m||B_m^nQ_mx||\leq s_m  k_n||x||,$
	\end{enumerate}
	for all $ (m,n,x)\in\Delta\times X$, where $(Q_n)$ is the complementary sequence of $(P_n)$.	
\end{theorem}
\begin{proof}
	\textit{Necessity:}We assume that the pair $(\mathcal{A},P)$ is $(h,k)-d.$ with $(P_n)$ strongly invariant for $(\mathcal{A})$. We have that:
	\begin{align*}
	h_m\|A_m^nP_nx\|\leq d_n h_n\|P_nx\|\leq d_n\|P_n\|h_n\|x\|\leq d_n(\|P_n\|+\|Q_n\|)h_n\|x\|\leq s_n h_n\|x\|.
	\end{align*} 
	 By Remark \ref{rem1} we obtain the implication $(kd_1)\Rightarrow (kd_2)$ .
	 \begin{align*}
	k_m\|B_m^nQ_mx\|&=k_m\|Q_nB_m^nQ_mx\| \leq d_mk_n\|A_m^nQ_nB_m^nQ_mx\|=d_m k_n\|Q_mx\|\\
	&\leq d_m\|Q_m\|k_n\|x\|\leq d_m(\|P_m\|+\|Q_m\|)k_n\|x\|\leq s_mk_n\|x\|,
	 \end{align*}
	 for all $(m,n,x)\in\Delta\times X$, where 
	 \begin{equation*}
	 s_n=\sup_{k\geq n}d_k(\|P_k\|+\|Q_k\|).
	 \end{equation*}
	 \textit{Sufficiency:}
	 The implication $(hd_2)\Rightarrow(hd_1)$ results by replacing $x$ with $P_nx$.
	 
	 The implication $(kd_2)\Rightarrow(kd_1)$ results by Remark \ref{rem1} for $d_n=s_n$.
	 \begin{align*}
	 k_m\|Q_nx\|=k_m\|B_m^nQ_mA_m^nQ_nx\|\leq s_m k_n\|A_m^nQ_nx\|
	 \end{align*}
	 \qed 
	\end{proof}
\begin{definition}\label{def-norma-compatibila}
	A sequence of norms $\mathcal{N}_1=\{\||\cdot\||_n,n\in\mathbb{N}\}$ is called \textit{compatible} with the sequence of projectors $(P_n) $ if there exists a nondecreasing sequence $c:\mathbb{N}\to[1,\infty)$ such that 
	\begin{equation}\label{norma-compatibila-cu-P}
	\|x\|\leq \||x\||_n\leq c_n(\|P_nx\|+\|Q_nx\|),
	\end{equation}
	for all $(n,x)\in\mathbb{N}\times X, $ where $(Q_n)$ is the complementary projectors sequence of $(P_n)$.
\end{definition}
\begin{remark}\label{rem-norma-compatibila-fara-proiectori}
	A sequence of norms $\mathcal{N}_1=\{\||\cdot\||_n,n\in\mathbb{N}\}$ is compatible with the sequence of projectors $(P_n) $ if	and only if there exists a nondecreasing sequence $c^1_n\geq 1$ such that
	\begin{equation}\label{norma-compatibila-tare}
	\|x\|\leq \||x\||_n\leq c^1_n\|x\|
	\end{equation}
	\end{remark}
\begin{proof}
	\textit{Necessity:} We only have to prove the right side of (\ref{norma-compatibila-cu-P}). By (\ref{norma-compatibila-tare}) there exists a nondecreasing sequence $c^1_n\geq 1$ such that
	\begin{align*}
	\||x\||\leq c^1_n\|x\|=c^1_n\|P_nx+Q_nx\|\leq c^1_n(\|P_nx\|+\|Q_nx\|),
	\end{align*}for all $(n,x)\in\mathbb{N}\times X$.
	
	\textit{Sufficiency:} From inequalities (\ref{norma-compatibila-cu-P}) there exists a nondecreasing sequence $c_n\geq 1$ such that
	\begin{align*}
	\|x\|\leq\||x\||_n\leq c_n(\|P_nx\|+\|Q_nx\|)\leq c_n(\|P_n\|+\|Q_n\|)\|x\|\leq c^1_n\|x\|,
	\end{align*}for all $(n,x)\in\mathbb{N}\times X$, where $c^1_n=\sup_{k\leq n}c_k(\|P_k\|+\|Q_k\|)$.
	\qed 
	\end{proof}
\begin{example}\label{ex-norma-compatibila-dichotomie}
	If the pair $(\mathcal{A},P)$ is $(h,k)-d.$ with $(P_n)$ strongly invariant for $(\mathcal{A})$, then the sequence of norms $\mathcal{N}_1=\{\||\cdot\||_n,n\in\mathbb{N}\}$ given by
	\begin{equation}\label{norma-dicho-dichotomie}
	\||x\||_n=\sup_{m\geq n}\frac{h_m}{h_n}\|A_m^n P_n x\|+\sup_{p\leq n}\frac{k_n}{k_p}\|B_n^p Q_n x\|,
	\end{equation}
	for all $(n,x)\in\mathbb{N}\times X$ is compatible with $(P_n)$.
	
	Indeed taking $p=m=n$ in (\ref{norma-dicho-dichotomie}) we have that
	\begin{align*}
	\||x\||_n\geq \|P_nx\|+\|Q_n x\|\geq \|P_nx+Q_nx\|=\|x\|
	\end{align*}
	From Theorem \ref{thm-dicho cu strong crestere} there exists a nondecreasing sequence $s:\mathbb{N}\to[1,\infty)$ such that the inequalities $(hd_2),(kd_2)$ hold. If we replace $x$ by $P_nx$ in $(hd_2)$, respectively by $Q_nx$ in $(kd_2)$  we obtain that
	\begin{align*}
	h_m\|A_m^nP_nx\|&\leq s_nh_n\|P_nx\|\\
	k_n\|B_n^pQ_nx\|&\leq s_n k_p\|Q_nx\|,
	\end{align*}
	for all $m\geq n\geq p\geq 0$ and thus we have
	\begin{align*}
	\||x\||_n\leq s_n\|P_nx\|+s_n\|Q_nx\|,
	\end{align*}
	for all $(m,x)\in\mathbb{N}\times X.$
	
	In conclusion the Definition \ref{def-norma-compatibila} is satisfied for $c_n=s_n.$
\end{example}
Further we will give a characterization of the concept of $(h,k)-d.$ using  sequences of norms compatible with $(P_n)$, which is an equivalence between the $(h,k)-d.$ and a certain type of $u.-(h,k)-d.$ In general the uniform and nonuniform concepts are not the same (see Example \ref{ex-unif-care-nu-este-dicho}).
\begin{theorem}\label{teorem-unif=neunif}
	Let $(P_n)$ be strongly invariant for the discrete system $(\mathcal{A})$.
	The pair $(\mathcal{A},P)$ is $(h,k)-d.$ if and only if there exists a sequence of norms
	
	 $\mathcal{N}_1=\{\||\cdot\||_n,n\in\mathbb{N}\}$ compatible with $(P_n)$  such that
	\begin{enumerate}
		\item[$(hd_3)$]  $h_m\||A_m^nP_nx\||_m\leq   h_n\||P_nx\||_n$
		\item[$(kd_3)$] $ k_m\||B_m^nQ_mx\||_n\leq k_n\||Q_mx\||_m,$
	\end{enumerate}
	for all $ (m,n,x)\in\Delta\times X$, where $(Q_n)$ is the complementary sequence of $(P_n)$.
\end{theorem}
\begin{proof}
	\textit{Necessity:}
	
	Let $(\mathcal{A},P)$ be a pair which is $(h,k)-d.$ with $(P_n)$ strongly invariant for $(\mathcal{A})$. By Example \ref{ex-norma-compatibila-dichotomie} there exists a sequence of norms $\mathcal{N}_1=\{\||\cdot\||_n,n\in\mathbb{N}\}$, given by (\ref{norma-dicho-dichotomie}), compatible with $(P_n)$. We only have to prove the two inequalities. By the fact that $(P_n),(Q_n)$ are orthogonal, replacing $x$ by $P_nx$ or $Q_nx$ in (\ref{norma-dicho-dichotomie}) and since $[m,\infty)\subseteq[n,\infty), [0,n]\subseteq[0,m]$ for all $(m,n)\in \Delta$ we obtain that:
	\begin{align*}
	\||A_m^n P_nx\||_m&=\||P_mA_m^nP_nx\||_m=\sup_{j\geq m}\frac{h_j}{h_m}\|A_j^mP_mA_m^nP_nx\|=\sup_{j\geq m}\frac{h_j}{h_m}\|A_j^nP_nx\|\\
	&\leq \frac{h_n}{h_m}\sup_{j\geq n}\frac{h_j}{h_n}\|A_j^nP_nx\|=\frac{h_n}{h_m}\||P_nx\||_n
	\end{align*}
	and
	\begin{align*}
	\||B_m^n Q_mx\||_n&=\||Q_nB_m^nQ_mx\||_n=\sup_{p\leq n}\frac{k_n}{k_p}\|B_n^pQ_nB_m^nQ_mx\|=\sup_{p\leq n}\frac{k_n}{k_p}\|B_m^nQ_mx\|\\
	&\leq \frac{k_n}{k_m}\sup_{p\leq m}\frac{k_m}{k_p}\|B_m^nQ_mx\|=\frac{k_n}{k_m}\||Q_mx\||_m,
	\end{align*}
	for all $(m,n,x)\in \Delta\times X.$
	
	\textit{Sufficiency:}
	
	We assume that there exists a sequence of norms $\mathcal{N}_1=\{\||\cdot\||_n,n\in\mathbb{N}\}$ compatible with $(P_n)$  such that the inequalities $(hd_3),(kd_3)$ hold. According to the Theorem \ref{thm-dicho cu strong crestere} we only have to prove that the inequalities $(hd_2),(kd_2)$ hold. By Definition \ref{def-norma-compatibila} we have that there exists a nondecreasing sequence $c_n\geq 1$ such that (\ref{norma-compatibila-cu-P}) are satisfied. Thus we obtain that:
	\begin{align*}
	h_m\|A_m^nP_nx\|&\leq h_m\||A_m^nP_nx\||_m\leq h_n\||P_nx\||_n\leq c_nh_n \|P_n x\|\\
	&\leq c_n \|P_n\|h_n\|x\|\leq c_n(\|P_n\|+\|Q_n\|)h_n\|x\|\leq s_nh_n\|x\|\\
	\text{and}\\
	k_m\|B_m^nQ_mx\|&\leq k_m\||B_m^nQ_mx\||_n\leq k_n\||Q_mx\||_m\leq c_m k_m\|Q_mx\|\\
	&\leq c_m \|Q_m\|k_n\|x\|\leq c_m(\|P_m\|+\|Q_m\|)k_n\|x\|\leq s_mk_n\|x\|
	\end{align*}
	for all $(m,n,x)\in\Delta\times X$, where $s_n=\sup_{k\leq n}c_k(\|P_k\|+\|Q_k\|).$
	\qed.
\end{proof}
\begin{theorem}\label{thm-unif-neunif-fara-proiectori-in-dreapta}
		Let $(P_n)$ be strongly invariant for the discrete system $(\mathcal{A})$.
	The pair $(\mathcal{A},P)$ is $(h,k)-d.$ if and only if there exists a sequence of norms
	
	 $\mathcal{N}_1=\{\||\cdot\||_n,n\in\mathbb{N}\}$ compatible with $(P_n)$  such that
	\begin{enumerate}
		\item[$(hd_4)$]  $h_m\||A_m^nP_nx\||_m\leq   h_n\||x\||_n$
		\item[$(kd_4)$] $ k_m\||B_m^nQ_mx\||_n\leq k_n\||x\||_m,$
	\end{enumerate}
	for all $ (m,n,x)\in\Delta\times X$, where $(Q_n)$ is the complementary projectors sequence of $(P_n)$.
\end{theorem}
\begin{proof}\textit{Necessity:}
	It results by Theorem \ref{teorem-unif=neunif} and by the fact that
	\begin{equation*}
	\||P_nx\||_n=\sup_{m\geq n}\frac{h_m}{h_n}\|A_m^nP_nx\|\leq \sup_{m\geq n}\frac{h_m}{h_n}\|A_m^nP_nx\|+\sup_{p\leq n}\frac{k_n}{k_p}\|B^n_pQ_nx\|=\||x\||_n
	\end{equation*}
	and also
	\begin{equation*}
\||Q_nx\||_n=\sup_{p\leq n}\frac{k_n}{k_p}\|B^n_pQ_nx\|\leq \sup_{m\geq n}\frac{h_m}{h_n}\|A_m^nP_nx\|+\sup_{p\leq n}\frac{k_n}{k_p}\|B^n_pQ_nx\|=\||x\||_n,
\end{equation*}	for all $(n,x)\in\mathbb{N}\times X.$

\textit{Sufficiency:} Results replacing $x$ by $P_nx$, respectively by $Q_mx$ in $(hd_4),(kd_4)$.
\qed 
	\end{proof}
In the particular cases when the growth rates are exponential or polynomial we obtain the characterizations of $e.d$ and $p.d$ in terms of $u.e.d$ and $u.p.d$ with Lyapunov type sequences of norms.
\begin{corollary}
	Let $(P_n)$ be strongly invariant for the discrete system $(\mathcal{A})$.
	The pair $(\mathcal{A},P)$ is $e.d.$ if and only if there exist a sequence of norms $\mathcal{N}_1=\{\||\cdot\||_n,n\in\mathbb{N}\}$ compatible with $(P_n)$ and the real constants $\alpha,\beta>0$ such that
	\begin{enumerate}
		\item[$(ed_1)$] $\||A_m^nP_nx\||_m\leq  e^{-\alpha(m-n)}\||P_nx\||_n$
		\item[$(ed_2)$] $ \||B_m^nQ_mx\||_n\leq  e^{-\beta(m-n)}\||Q_mx\||_m,$
	\end{enumerate}
	for all $ (m,n,x)\in\Delta\times X$, where $(Q_n)$ is the complementary sequence of $(P_n)$.	
\end{corollary}
\begin{corollary}
	Let $(P_n)$ be strongly invariant for the discrete system $(\mathcal{A})$.
	The pair $(\mathcal{A},P)$ is $p.d.$ if and only if there exist a sequence of norms $\mathcal{N}_1=\{\||\cdot\||_n,n\in\mathbb{N}\}$ compatible with $(P_n)$ and the real constants $\alpha,\beta>0$ such that
	\begin{enumerate}
		\item[$(pd_1)$] $(m+1)^\alpha\||A_m^nP_nx\||_m\leq   (n+1)^\alpha\||P_nx\||_n$
		\item[$(pd_2)$] $ (m+1)^\beta\||B_m^nQ_mx\||_n\leq  (n+1)^\beta\||Q_mx\||_m,$
	\end{enumerate}
	for all $ (m,n,x)\in\Delta\times X$, where $(Q_n)$ is the complementary sequence of $(P_n)$.
\end{corollary}
\begin{definition}\label{def-crestere}
We say the pair $(\mathcal{A},P)$ has \textit{$(h,k)$- growth} (and we denote by $(h,k)-g.$) if there exists a nondecreasing sequence $g:\mathbb{N}\to[1,\infty)$  such that
\begin{enumerate}
	\item[$(hg_1)$] $h_n || A_m^nP_nx|| \leq g_n h_m  ||P_nx||$
	\item[$(kg_1)$] $ k_n || Q_nx|| \leq g_m  k_m ||A_m^nQ_nx||$,
\end{enumerate}
for all $ (m,n,x)\in\Delta\times X$, where $(Q_n)$ is the complementary sequence of $(P_n).$
\end{definition}
\begin{remark}
As particular cases of $(h,k)$-growth we have:
\begin{enumerate}
	\item If the sequence $(d_n)$ is constant then we obtain the uniform-$(h,k)$-growth property, denoted by $u.-(h,k)-g.$
	\item For $h_m=e^{\alpha m}, k_m=e^{\beta m} $ where $\alpha,\beta>0$ it  results the  exponential growth concept, denoted by $e.g.$
	\item If $h_m=(m+1)^\alpha, k_m=(m+1)^\beta,$ where $\alpha,\beta>0$ then we obtain the polynomial growth denoted by $p.g.$
\end{enumerate}	
\end{remark}

\begin{remark}
	If the pair $(\mathcal{A},P)$ is $(h,k)-d.$ then it has $(h,k)-g.$. The following example will emphases that the reverse is not always true.
\end{remark}
Let \begin{equation*}
\mathcal{H}=\{h,a:\mathbb{N}\to[1,\infty)|\lim\limits_{n\to\infty}\frac{h_n^2}{\ln a_n +1}=\infty \}.
\end{equation*}
\begin{example}\label{ex-are-crestere-nu-este-dicho}
	We consider the Banach space $X=\mathbb{R}^2$ endowed with the norm $\|(x_1,x_2)\|=\max\{|x_1|,|x_2|\}$	and let the projectors sequences $P,Q:\mathbb{N}\to\mathcal{B}(X)$ given by Example \ref{ex-unif-care-nu-este-dicho}.
	We consider the dynamical system $(\mathcal{A})$ generated by the sequence  $A_n$ given by
	\begin{equation*}
	A_n= \frac{1+\ln a_{n}}{1+\ln a_{n+1}}\left( \frac{h_n+1}{h_{n}}P_n+\frac{k_{n}}{k_n+1}Q_{n+1}\right) ,
	\end{equation*}
	with $h,a\in\mathcal{H}$ and $k:\mathbb{N}\to[1,\infty)$ a growth rate. 
	
	It follows that the evolution operator associated to the system  $(\mathcal{A})$ is given by
	\begin{eqnarray*}
		A_m^n= \frac{1+\ln a_n}{1+\ln a_m } \left( \frac{h_m}{h_{n}}P_n+\frac{k_{n}}{k_m}Q_{m}\right) ,&\text{for all }&(m,n)\in\Delta.
	\end{eqnarray*}
	We observe that
	\begin{eqnarray*}
		\|A_m^nP_nx\|=\frac{1+\ln a_n}{1+\ln a_m} \cdot\frac{h_m}{h_{n}}\|P_nx\|&\text{ and }& \|A_m^nQ_nx\|=\frac{1+\ln a_n}{1+\ln a_m }\cdot\frac{k_{n}}{k_m}\|Q_{n}x\|.
	\end{eqnarray*}
	It follows that Definition \ref{def-crestere} is satisfied for $g_n=1+\ln a_n.$
	
	If we assume that the pair $(\mathcal{A},P)$ is $(h,k)-d.$ then there exists a nondecreasing sequence $d_n\geq 1$ such that:
	\begin{align*}
	\frac{h_m^2}{1+\ln a_m }\leq d_n\frac{h_{n}^2}{1+\ln a_n}, \text{ for all } (m,n)\in\Delta.
	\end{align*}
	Taking $n=0$ and $m\to\infty$, since $h,a\in\mathcal{H}$ we have
	\begin{align*}
	\infty\leq d_0\frac{h_{0}^2}{1+\ln a_0},
	\end{align*}	which is absurd.
	
\end{example}
\begin{remark}
	If the pair $(\mathcal{A},P)$ has $u.-(k,h)-g.$ then it has $(h,k)-g.$. The reverse of this statement is not always valid.
	\end{remark}
\begin{example}
	We consider the pair $(\mathcal{A},P)$ given by the Example \ref{ex-are-crestere-nu-este-dicho} and we have that it has $(h,k)-g.$
	
	If we assume that the pair has $u.-(h,k)-g.$ then there exists a constant $M\geq 1$ such that
	\begin{eqnarray*}
	1+\ln a_{m}\leq M(1+\ln a_n), \text{for all } (m,n)\in\Delta.
	\end{eqnarray*}
Taking $n=0$ and $m\to\infty$ we obtain a contradiction.
	\end{example}

The following  theorem gives a characterization of the $(h,k)$-growth property with strongly invariant projectors. 
\begin{theorem}\label{thm-crestere-cu-tare-inv}
Let $(P_n)$ be strongly invariant for $(\mathcal{A})$.
The pair $(\mathcal{A},P)$ has $(h,k)$- growth if there exists a nondecreasing sequence $ g:\mathbb{N}\rightarrow [1,\infty)$ such that
\begin{enumerate}
	\item[$(hg_2)$] $h_n||A_m^nP_nx||\leq g_n  h_m||x||$
	\item[$(kg_2)$] $ k_n||B_m^nQ_mx||\leq g_m  k_m||x||,$
\end{enumerate}
for all $ (m,n,x)\in\Delta\times X$, where $(Q_n)$ is the complementary projectors sequence of $(P_n)$.
\end{theorem}
\begin{proof}
Similar to the proof of Theorem \ref{thm-dicho cu strong crestere}.	
\qed 
\end{proof}
We will give another example of a sequence of norms compatible with $(P_n)$.
\begin{example}\label{ex-norma-compatibila-crestere}
If the pair $(\mathcal{A},P)$ has $(h,k)-g.$ then the sequence of norms

 $\mathcal{N}=\{\|\cdot\|_n,n\in\mathbb{N}\}$ given by
\begin{equation}\label{norma-dicho-crestere}
\|x\|_n=\sup_{m\geq n}\frac{h_n}{h_m}\|A_m^n P_n x\|+\sup_{p\leq n}\frac{k_p}{k_n}\|B_n^p Q_n x\|,
\end{equation}
for all $(n,x)\in\mathbb{N}\times X$ is compatible with $(P_n)$.

Indeed taking $p=m=n$ in (\ref{norma-dicho-crestere}) we have that
\begin{align*}
\|x\|_n\geq \|P_nx\|+\|Q_n x\|\geq \|P_nx+Q_nx\|=\|x\|
\end{align*}
From Theorem \ref{thm-crestere-cu-tare-inv} there exists a nondecreasing sequence $g:\mathbb{N}\to[1,\infty)$ such that the inequalities $(hg_2),(kg_2)$ hold. Thus we obtain that:
\begin{align*}
\|x\|_m\leq g_m\|x\|+g_m\|x\|,
\end{align*}
for all $(m,x)\in\mathbb{N}\times X.$

We obtain from Remark \ref{rem-norma-compatibila-fara-proiectori} that the sequence of norms $\mathcal{N}$ is compatible with $(P_n).$
\end{example}
Further we will give another characterization of the concept of $(h,k)$- dichotomy with sequences of norms compatible with $(P_n)$.
\begin{theorem}\label{torema-2}
Let $(P_n)$ be strongly invariant for the discrete system $(\mathcal{A})$.
The pair $(\mathcal{A},P)$ is $(h,k)-d.$ if and only if there exist a sequence of norms $\mathcal{N}=\{\|\cdot\|_n,n\in\mathbb{N}\}$ compatible with $(P_n)$ and a nondecreasing sequence $s:\mathbb{N}\to[1,\infty)$ such that
\begin{enumerate}
	\item[$(hd_5)$] $h_m\|A_m^nP_nx\|_m\leq s_n  h_n\|x\|_n$
	\item[$(kd_5)$] $ k_m\|B_m^nQ_mx\|_n\leq s_m k_n\|x\|_m,$
\end{enumerate}
for all $ (m,n,x)\in\Delta\times X$, where $(Q_n)$ is the complementary sequence of $(P_n)$.
\end{theorem}
\begin{proof}
\textit{Necessity:}
It results from Theorem \ref{thm-unif-neunif-fara-proiectori-in-dreapta} for $s_n=1.$

\textit{Sufficiency:}
By (\ref{norma-compatibila-tare}) and the inequalities $(hd_5),(kd_5)$ there exist two nondecreasing sequences $(s_n),(c^1_n)\geq 1$ such that:
\begin{equation*}
\|A_m^nP_nx\|\leq \|P_m A_m^n P_nx\|_m\leq s_n\frac{h_n}{h_m}\|x\|_n\leq s_n c^1_n\frac{h_n}{h_m}\|x\|
\end{equation*}
and
\begin{equation*}
\|B_m^nQ_mx\|\leq \|Q_n B_m^n Q_mx\|_n\leq s_m\frac{k_n}{k_m}\|x\|_m\leq s_m c^1_m\frac{k_n}{k_m}\|x\|.
\end{equation*}
Thus the pair $(\mathcal{A},P)$ is $(h,k)-d.$
\qed 
\end{proof}
\begin{corollary}
Let $(P_n)$ be strongly invariant for the discrete system $(\mathcal{A})$.
The pair $(\mathcal{A},P)$ is $e.d.$ if and only if there exist a sequence of norms $\mathcal{N}=\{\|\cdot\|_n,n\in\mathbb{N}\}$ compatible with $(P_n)$, the real constants $\alpha,\beta>0$ and a nondecreasing map $s:\mathbb{N}\to[1,\infty)$ such that
\begin{enumerate}
	\item[$(ed_1')$] $\|A_m^nP_nx\|_m\leq s_n e^{-\alpha(m-n)}\|x\|_n$
	\item[$(ed_2')$] $ \|B_m^nQ_mx\|_n\leq s_m e^{-\beta(m-n)}\|x\|_m,$
\end{enumerate}
for all $ (m,n,x)\in\Delta\times X$, where $(Q_n)$ is the complementary sequence of $(P_n)$.	
\end{corollary}
\begin{corollary}
Let $(P_n)$ be strongly invariant for the discrete system $(\mathcal{A})$.
The pair $(\mathcal{A},P)$ is $p.d.$ if and only if there exist a sequence of norms $\mathcal{N}=\{\|\cdot\|_n,n\in\mathbb{N}\}$ compatible with $(P_n)$, the real constants $\alpha,\beta>0$ and a nondecreasing map $s:\mathbb{N}\to[1,\infty)$ such that
\begin{enumerate}
	\item[$(pd_1')$] $(m+1)^\alpha\|A_m^nP_nx\|_m\leq  s_n (n+1)^\alpha\|x\|_n$
	\item[$(pd_2')$] $ (m+1)^\beta\|B_m^nQ_mx\|_n\leq s_m (n+1)^\beta\|x\|_m,$
\end{enumerate}
for all $ (m,n,x)\in\Delta\times X$, where $(Q_n)$ is the complementary sequence of $(P_n)$.
\end{corollary}
\section{Applications to characterizations of $(h,k)-$ dichotomy with sequences of norms }
	The aim of this section is to provide a characterization of the concept with Barbashin type theorem and necessary and sufficient conditions of Datko type (for exponential case see \cite{popa} and the reference therein), with sequences of norms. 
	
	We denote by $\mathcal{\tilde{H}}$ the set of growth rates $h:\mathbb{N}\to[1,\infty)$ for which there exist a sequence $\tilde{h}:\mathbb{N}\to(0,\infty)$ and a constant $H\geq 1$ such that 
	\begin{equation}
	\sum_{n=0}^{\infty}\frac{\tilde{h}_n}{h_n}\leq H.
	\end{equation} 
	It is easy to see that if $h$ is an exponential rate (i.e. $h_n=e^{\alpha n}$ with $\alpha>0$) or a polynomial rate (i.e. $h_n=(n+1)^\alpha$ with $\alpha>0$) then $h\in\mathcal{\tilde{H}}$.
	
	A characterization of Barbashin type for $(h,k)-$ dichotomy in terms of Lyapunov sequences is given by
\begin{theorem}
	Let $P:\mathbb{N}\to\mathcal{B}(X)$ be a sequence of projectors strongly invariant for the system $(\mathcal{A})$. The pair $(\mathcal{A},P)$ is $(h,k)-d.$ if and only if there exist a sequence of norms $\mathcal{N}_1=\{\||\cdot\||_n,n\in\mathbb{N}\}$ compatible with $(P_n)$ and a nondecreasing sequence $b:\mathbb{N}\to[1,\infty)$ such that:
	\begin{enumerate}
		\item [$({h}d_6)$]
		\begin{equation*}
		\sum_{j=0}^{n}h_m\||A_m^jP_jx\||_m\leq b_nh_n\||x\||_n,
		\end{equation*}
		\item [$(kd_6)$]
		\begin{equation*}
		\sum_{j=0}^{n}\frac{\||B_m^jQ_mx\||_j}{k_j}\leq b_m\frac{\||x\||_m}{k_m},
		\end{equation*}
		for all $(m,n,x)\in\Delta\times X.$
	\end{enumerate}
\end{theorem}
\begin{proof}\textit{Necessity:}
If we assume that the pair $(\mathcal{A},P)$ is $(h,k)-d.$ with $(P_n)$ strongly invariant for $(\mathcal{A})$ then by Example \ref{ex-norma-compatibila-dichotomie} there exists a sequence of norms 

$\mathcal{N}_1=\{\||\cdot\||_n,n\in\mathbb{N}\}$ compatible with $(P_n)$ and by Theorem \ref{thm-unif-neunif-fara-proiectori-in-dreapta}  the inequalities $(hd_4),(kd_4)$ are satisfied. By Remark \ref{norma-compatibila-tare} there exists a nondecreasing sequence $c^1_n\geq 1$ such that:
\begin{align*}
\sum_{j=0}^{n}{h}_m\||A_m^jP_jx\||_m&\leq\sum_{j=0}^{n}{h}_j\||x\||_j\leq h_n \sum_{j=0}^{n}c^1_j\|x\|
\leq (n+1)c^1_n h_n\|x\|\leq b_nh_n\||x\||_n\\
\text{  and  }\\
\sum_{j=0}^{n}\frac{\||B_m^jQ_mx\||_j}{k_j}&\leq(n+1)\frac{\||x\||_m}{k_m}\leq(m+1)\frac{\||x\||_m}{k_m}\leq b_m\frac{\||x\||_m}{k_m},
\end{align*}
for all $(n,x)\in\mathbb{N}\times X$, where $b_n=(1+n)c_n^1\geq 1$.

\textit{Sufficiency:}	It results for $j=n$ in $(hd_6),(kd_6)$ and Theorem \ref{torema-2}.
\qed 
	\end{proof}
A necessary condition of Datko-type for $(h,k)-$ dichotomy is
\begin{theorem}
			Let $P:\mathbb{N}\to\mathcal{B}(X)$ be a sequence of projectors strongly invariant for the system $(\mathcal{A})$. If the pair $(\mathcal{A},P)$ is $(h,k)-d.$ with $h\in\mathcal{H}$ then there exist a sequence of norms $\mathcal{N}_1=\{\||\cdot\||_n,n\in\mathbb{N}\}$ compatible with $(P_n)$ and a nondecreasing sequence $d:\mathbb{N}\to[1,\infty)$ such that:
	\begin{enumerate}
		\item [$({\tilde{h}}d_7)$]
		\begin{equation*}
		\sum_{j=n}^{m}\tilde{h}_j\||A_j^nP_nx\||_j\leq d_n\tilde{h}_n\||x\||_n,
		\end{equation*}
		\item [$({k}d_7)$]
		\begin{equation*}
		\sum_{j=n}^{m}{k}_j\||B_j^nQ_jx\||_n\leq d_m{k}_n\||x\||_m,
		\end{equation*}
		for all $(m,n,x)\in\Delta\times X.$
	\end{enumerate}
\end{theorem}
\begin{proof}
	By Theorem \ref{thm-unif-neunif-fara-proiectori-in-dreapta} and the inequalities (\ref{norma-compatibila-tare}) it follows that there exists a nondecreasing sequence $c_n^1\geq 1$ such that
\begin{align*}
\sum_{j=n}^{m}\tilde{h}_j\||A_j^nP_nx\||_j&\leq {h_n}\||x\||_n \sum_{j=n}^{m}\frac{\tilde{h}_j}{h_j}\leq H \frac{h_n}{\tilde{h}_n}\tilde{h}_n \||x\||_n
\leq d_n {h}_n\||x\||_n\\
\text{and}\\
\sum_{j=n}^{m}{k}_j\||B_j^nQ_jx\||_n&\leq \sum_{j=n}^{m}{k_n}\||x\||_j\leq k_n \sum_{j=n}^{m} c_j^1\|x\| \leq c^1_m(D+1)k_n\|x\|\\
&\leq d_m {k}_n\||x\||_m,
\end{align*}	for all $(m,n,x)\in\Delta\times X,$ where $d_n=H(c_n^1+\frac{h_n}{\tilde{h}_n}).$\qed

\end{proof}
A sufficient condition of Datko-type for $(h,k)-$ dichotomy is
\begin{theorem}
		Let $P:\mathbb{N}\to\mathcal{B}(X)$ be a sequence of projectors strongly invariant for the system $(\mathcal{A})$. If there exist a sequence of norms $\mathcal{N}_1=\{\||\cdot\||_n,n\in\mathbb{N}\}$ compatible with $(P_n)$ and a nondecreasing sequence $d:\mathbb{N}\to[1,\infty)$ such that the inequalities $(hd_7)$ and $(kd_7)$ are satisfied then the pair $(\mathcal{A},P)$ is $(h,k)-d.$
\end{theorem}
\begin{proof}
It results taking $j=m$ in $(hd_7),(kd_7)$ and by Theorem \ref{torema-2}.
\qed 
	\end{proof}
\begin{acknowledgement}
The authors would like to express their deep gratitude to Professor Emeritus Mihail Megan  for his valuable and constructive suggestions and  useful critiques during the planning and development of this research work. 
\end{acknowledgement}
%%%%%%%%%%%%%%%%%%%%%%%% referenc.tex %%%%%%%%%%%%%%%%%%%%%%%%%%%%%%
% sample references
% %
% Use this file as a template for your own input.
%
%%%%%%%%%%%%%%%%%%%%%%%% Springer-Verlag %%%%%%%%%%%%%%%%%%%%%%%%%%
%
% BibTeX users please use
% \bibliographystyle{}
% \bibliography{}
%

\end{document}